\documentclass[preprint,12pt]{elsarticle}

\usepackage{hyperref}

\usepackage{amssymb}
\usepackage{amsthm}
\usepackage{epsfig}
\usepackage{amsmath}

\theoremstyle{plain}
\newtheorem{thm}{Theorem}[section]

\newtheorem{lem}[thm]{Lemma}

\journal{Computers and Mathematics with Applications}

\begin{document}

\begin{frontmatter}



\title{Discontinuous Galerkin Method for the Air Pollution Model}


\author[SAP]{Lite Zhao}
\ead{zhaolite@gmail.com}
\author[SCM]{Xijian Wang\corref{cor1}}
\ead{wangxj1980426@gmail.com}
\cortext[cor1]{Corresponding author}
\author[TUE]{Qinzhi Hou}
\ead{q.hou@tue.nl}
\address[SAP]{School of Applied Physics and Materials, Wuyi University, People's Republic of China}
\address[SCM]{School of Mathematics and Computing Science, Wuyi University, People's Republic of China}
\address[TUE]{Department of Mathematics and Computer Science, Eindhoven University of Technology, The Netherlands}
\begin{abstract}
In this paper we present the discontinuous Galerkin method to solve the problem of the two-dimensional air pollution model. The resulting system of ordinary differential equations is called the semidiscrete formulation. We show the existence and uniqueness of the ODE system and provide the error estimates for the numerical error.

\end{abstract}


\end{frontmatter}


\section{Introduction}
\label{}
Air pollution is the introduction of chemicals, particulate matter, or biological materials that cause harm or discomfort to humans or other living organisms, or cause damage to the natural environment or built environment, into the atmosphere. The basic technology for analyzing air pollution is through the mathematical models and numerical methods for predicting the transport of air pollutants in the lower atmosphere\cite{daly2007air, blomcomparison, botchev2003new, dimov2004computational, zlatev1995computer}. Different air pollution models have been developed in the last decades by the National Environmental Research Institute (\url{http://www.dmu.dk/en/air/models/}). In the present paper we consider the following Danish Eulerian model \cite{blomcomparison,dimov2004computational,zlatev1995computer}

\begin{subequations}
\begin{align}\label{Eulerianmodela}
  &\frac{{\partial u}}{{\partial t}} +\frac{\partial }{{\partial x}}(cu) + \frac{\partial }{{\partial y}}(eu)-\frac{\partial }{{\partial x}}({k_x}\frac{{\partial u}}{{\partial x}}) - \frac{\partial }{{\partial y}}({k_y}\frac{{\partial u}}{{\partial y}})=f(u) , \\\label{Eulerianmodelb}
  &f(u)= - ({k_1} + {k_2})u+ E + Q(u),\\\label{Eulerianmodel1c}
  &u(x,\;y,\;0) = {u_0}(x,\;y),\;\;(x,\;y) \in \Omega , \\\label{Eulerianmodeld}
  &u(x,\;y,\;t)\left| {_{\partial \Omega }} \right. = 0,\;\;t \in [0,\;T].
\end{align}
\end{subequations}

The different quantities involved in the mathematical model have the following meaning:
\begin{itemize}
  \item the concentration is denoted by $u$;
  \item $c$ and $e$ are wind velocities;
  \item $k_{x}$ and $k_{y}$ are diffusion coefficients;
  \item the emission source is described by $E$;
  \item $k_{1}$ and $k_{2}$ are constant deposition coefficients;
  \item the chemical reaction is denoted by $Q$.
\end{itemize}
Meanwhile, we give the following assumptions:

 \begin{itemize}
  \item $u\in H_{0}^{1}(\Omega)\cap H^{3}(\Omega),\;\; u_{t}, u_{tt}\in L^{2}(\Omega)$;
  \item $Q(u)$ satisfy the Lipschitz condition;
  \item $
   0 < {k_*} \leqslant \min \{ \left| {{k_x}} \right|,\;\left| {{k_y}} \right|\}  \leqslant \max \{ \left| {{k_x}} \right|,\;\left| {{k_y}} \right|\}  \leqslant {k^*},\;0 < {c_*} \leqslant \min \{ \left| c \right|,\;\left| e \right|\}  \leqslant
  \;\;\;\max \{ \left| c \right|,\;\left| e \right|\}  \leqslant {c^*},\;{k_*},\;{k^*},\;{c_*},\;{c^*}
 $are constants.
  \end{itemize}

A general description of the Danish Eulerian Model and its numerical treatment is given in \cite{zlatev1995computer,zlatev2006computational,zlatev1994studying}. Research on the finite difference method and finite volume element method for this air pollution model already has good results \cite{Yuan2000numerical,guo2007numerical,Zhou1997,wang10}. In this article, we use the discontinuous Galerkin method (DG method) to analyse and solve the air pollution model.

DG methods in mathematics form a class of numerical methods for solving partial differential equations. They have recently gained popularity due to many of their attractive properties, refer to \cite{Coc00, Riv08, arnold1982interior, baumann1999discontinuous, Oden:1998:DHF:301787.301788, riviere2000discontinuous, riviere1999improved,rivi¨¨re2002priori, wheeler1980interior}. First of all, the flexibility of the methods allows for general non-conforming meshes with variable degree of approximation. This makes the implementation of h-p adaptivity for DG easier than that for conventional approaches. Moreover, the DG methods are locally mass conservative at the element level. In addition, they have less numerical diffusion than most conventional algorithms, thus are likely to offer more accurate solution for at least advection-dominated transport problems. They handle rough coefficient problems and capture the discontinuity in the solution very well by the nature of discontinuous function space. Furthermore, the DG methods are easier to implement than most traditional finite element methods. The trial and test spaces are easier to construct than conforming methods because they are local.

The paper is organized as follows: In Section 2, the variational formulation of the DG method is stated. And we show the existence and uniqueness of the resulting ordinary differential equations system. Finally we provide the error estimates for the numerical error in Section 3.


\section{Semidiscrete formulation}
In this section, we approximate the solution $u(t)$ by a function $U_{h}(t)$ that belongs to the finite-dimensional space $\mathcal{D}_{k}(\varepsilon_{h})$ for all $t\geq 0$. The solution $U_{h}$ is referred to as the semidiscrete solution. In what follows, we assume that $s>\frac{3}{2}$. We introduce a bilinear form $J_{0}^{\sigma_{0},\beta_{0}}$:
$H^{s}({\varepsilon_{h}})\times H^{s}({\varepsilon_{h}})\rightarrow \mathbb{R} $ that penalize the jump of the function values:
\begin{equation*}
    J_{0}^{\sigma_{0},\beta_{0}}(w, v)=\sum\limits_{e\in \Gamma _{h}\cup\partial{\Omega}}\frac{\sigma^{0}_{e}}{\left| {e} \right|^{\beta_{0}}}\int_{e}{[w][v]}
\end{equation*}
The parameter $\sigma_{e}^{0}$ is called penalty parameter. It is nonnegative real number. The power $\beta_{0}$ is positive number. $\left| {e} \right|$ simply means the length of $e$. We now define the DG bilinear form $a_{\epsilon}:H^{s}({\varepsilon_{h}})\times H^{s}({\varepsilon_{h}})\rightarrow \mathbb{R}$
\begin{align*}
    a_{\epsilon}(w,v)=&\sum\limits_{E\in \varepsilon_{h}}\int_{E}( k_{x}\frac{{\partial w}}{{\partial x}}\frac{\partial v}{\partial x}+k_{y}\frac{{\partial w}}{{\partial y}}\frac{\partial v}{\partial y})-\sum\limits_{e\in \Gamma _{h}}\int_{e}(\{k_{x}\frac{{\partial w}}{{\partial x}}\overrightarrow{n_{1}}\}+\{k_{y}\frac{{\partial w}}{{\partial y}}\overrightarrow{n_{2}}\})[v]\\
      &-\epsilon \sum\limits_{e\in \Gamma _{h}}\int_{e} {(\{k_{x}\frac{{\partial v}}{{\partial x}}\overrightarrow{n_{1}}\}+\{k_{y}\frac{{\partial v}}{{\partial y}}\overrightarrow{n_{2}}\})[w]}+ J_{0}^{\sigma_{0},\beta_{0}}(w, v).\\
\end{align*}
The bilinear form $a_{\epsilon}$ contains another parameter $\epsilon$ that may take the value -1,0, or 1. $a_{\epsilon}$ is symmetric if $\epsilon=-1$ and it is nonsymmetric otherwise.

This bilinear form yields the following energy seminorm:
\begin{equation*}
    {\left\| v \right\|_\varepsilon } = {(\sum\limits_{E \in {\varepsilon _h}} {\left\| {{D^{1/2}}\nabla v} \right\|_{{L^2}(E)}^2}  + \sum\limits_{e\in \Gamma _{h}}\frac{\sigma^{0}_{e}}{\left| {e} \right|^{\beta_{0}}}\left\| [v] \right\|_{{L^2}(e)}^2)^{1/2}}
\end{equation*}
Second, the convection term $\frac{\partial }{{\partial x}}(cu) + \frac{\partial }{{\partial y}}(eu)$ is approximated by an upwind discretization. Let us denote the upwind value of a function $w$ by $w^{up}$. We recall that $\left( {\begin{array}{*{20}{c}}
  {\overrightarrow {{n_1}} } \\
  {\overrightarrow {{n_2}} }
\end{array}} \right)$ is a unit normal vector pointing from $E_{e}^{1}$ to $E_{e}^{2}$:
\[{w^{up}} = \left\{ \begin{gathered}
  w\left| {_{E_e^1}} \right.,\;\;if\;\; c\overrightarrow{n_{1}}+e\overrightarrow{n_{2}}\geq 0\;\; \hfill \\
  w\left| {_{E_e^2}} \right.,\;\;if\;\; c\overrightarrow{n_{1}}+e\overrightarrow{n_{2}}\leq 0\;\; \hfill \\
\end{gathered}  \right.\;\;\;\;\;\;\;\forall e =\partial E^{1}_{e}\cap \partial E^{2}_{e}. \]
Let
\begin{equation*}
    b(c,e;w,v)= -\sum\limits_{E\in \varepsilon_{h}}{\int_{E} {(cw\frac{\partial v}{\partial x} + ew\frac{\partial v}{\partial y})}}  + \sum\limits_{e\in \Gamma _{h}}{\int_{e} {(c\overrightarrow{n_{1}}w^{up}[v] + e\overrightarrow{n_{2}}w^{up}[v])}}
\end{equation*}

The general semidiscrete DG variational formulation of problem (\ref{Eulerianmodela})-(\ref{Eulerianmodeld}) is as follows: Find $U_{h}\in L^{2}(0, T; \mathcal{D}_{k}(\varepsilon_{h}))$, such that
\begin{subequations}
\begin{align}\label{variationalforma}
  &\forall t>0, \forall v\in \mathcal{D}_{k}(\varepsilon_{h}), (\frac{\partial U_{h}}{\partial t}, v)_{\Omega}+a_{\epsilon}(U_{h}(t),v)+b(c,e;U_{h}(t),v)=L(U_{h}(t),v) , \\
  \label{variationalformb}
  &\forall v\in \mathcal{D}_{k}(\varepsilon_{h}), (U_{h}(0),v)_{\Omega}=(u_{0},v)_{\Omega},
\end{align}
\end{subequations}
where the form $L$ is
\begin{equation*}
    L(w;v)=\int_{\Omega}f(w)v.
\end{equation*}

The next lemma establishes the consistency between the model problem and the variational formulation.
\begin{lem}
Assume that the weak solution $u$ of problem (\ref{Eulerianmodela})-(\ref{Eulerianmodeld}) belongs to $H^{1}(0, T; H^{2}(\varepsilon_{h}))$, then $u$ satisfies the variational problem (\ref{variationalforma})-(\ref{variationalformb}).
\end{lem}

\begin{proof}
Let $v$ be a test function in $\mathcal{D}_{k}(\varepsilon_{h})$. We multiply by $v|_{E}$ and integrate by parts on one element $E\in \varepsilon_{h}$, and use Green's theorem:
\begin{align*}
       &{(\frac{{\partial u}}{{\partial t}},\;v)_E} - \int_{E} {(-k_{x}\frac{{\partial u}}{{\partial x}}\frac{\partial v}{\partial x}-k_{y}\frac{{\partial u}}{{\partial y}}\frac{\partial v}{\partial y}+cu\frac{\partial v}{\partial x} + eu\frac{\partial v}{\partial y})}  + \\
       &\int_{\partial{E}} {(-k_{x}\frac{{\partial u}}{{\partial x}}\overrightarrow{n_{1}}v-k_{y}\frac{{\partial u}}{{\partial y}}\overrightarrow{n_{2}}v+cu\overrightarrow{n_{1}}v + eu\overrightarrow{n_{2}}v)} = \int_{E} {f(u)v}\\
\end{align*}
Summing over all elements and using the regularity of the exact solution, we obtain
\begin{align*}
       &{(\frac{{\partial u}}{{\partial t}},\;v)_\Omega} -\sum\limits_{E\in \varepsilon_{h}}{\int_{E} {(-k_{x}\frac{{\partial u}}{{\partial x}}\frac{\partial v}{\partial x}-k_{y}\frac{{\partial u}}{{\partial y}}\frac{\partial v}{\partial y}+cu\frac{\partial v}{\partial x} + eu\frac{\partial v}{\partial y})}}  + \\
       & \sum\limits_{e\in \Gamma _{h}}{\int_{e} {(-\{k_{x}\frac{{\partial u}}{{\partial x}}\overrightarrow{n_{1}}\}[v]-\{k_{y}\frac{{\partial u}}{{\partial y}}\overrightarrow{n_{2}}\}[v]+cu\overrightarrow{n_{1}}[v] + eu\overrightarrow{n_{2}}[v])}}+\\
      &\epsilon \sum\limits_{e\in \Gamma _{h}}\int_{e} {(-\{k_{x}\frac{{\partial v}}{{\partial x}}\overrightarrow{n_{1}}\}[u]-\{k_{y}\frac{{\partial v}}{{\partial y}}\overrightarrow{n_{2}}\}[u])}+ \sum\limits_{e\in \Gamma _{h}}\frac{\sigma^{0}_{e}}{\left| {e} \right|^{\beta_{0}}}\int_{e}{([u][v])}= \int_{\Omega} {f(u)v}.\\
\end{align*}
Since $u_{up}=u$, we clearly have our result.
\end{proof}
\subsection{Existence and uniqueness of the solution}
Because of the lack of continuity constraints between mesh elements for the test functions, the basic functions of $\mathcal{D}_{k}(\varepsilon_{h}))$ have a support contained in one element. We write
\begin{equation*}
    \mathcal{D}_{k}(\varepsilon_{h})=span\{ \phi _i^E:\;1 \leqslant i \leqslant {N_{loc}},\;E \in {\varepsilon _h}\}
\end{equation*}
with
\[\phi _i^E(x) = \left\{ \begin{gathered}
  \widetilde {{\phi _i}} \circ {F_E}(x),\;\;\;x \in E, \hfill \\
  0,\;\;\;\;\;\;\;\;\;\;\;\;\;\;x \notin E. \hfill \\
\end{gathered}  \right.\]
In 2D, we have $\widehat{\phi}(\widehat{x},\widehat{y})=\widehat{x}^{I}\widehat{y}^{I}, I+J=i, 0\leq i\leq k$. This yields the local dimension
\begin{equation*}
    N_{loc}=\frac{(k+1)(k+2)}{2}.
\end{equation*}
using the global basis functions, we can expand the semidiscrete solution
\begin{equation}\label{expansion}
    \forall t \in (0,\;T),\;\forall (x,\;y) \in \Omega ,\;{U_h}(t,\;x,\;y) = \sum\limits_{E \in {\varepsilon _h}} {\sum\limits_{i = 1}^{{N_{loc}}} {\xi _i^E(t)\phi _i^E(x,\;y)} }.
\end{equation}
The degrees of freedom $\xi^{E}$'s are functions of time. Let $N_{el}$ denote the number of elements in the mesh. We can rename the basis functions and the degrees of freedom such that
\begin{align*}
    & \{\phi_{i}^{E}:1\leq i \leq N_{loc}, E\in \varepsilon_{h}\}=\{\widetilde{\phi}_{j}:1\leq j \leq N_{loc}N_{el}\},\\
    & \{\xi_{i}^{E}:1\leq i \leq N_{loc}, E\in \varepsilon_{h}\}=\{\widetilde{\xi}_{j}:1\leq j \leq N_{loc}N_{el}\}.
\end{align*}
Plugging (\ref{expansion}) into the variational problem (\ref{variationalforma})-(\ref{variationalformb}) yields a linear system of ordinary differential equations as follows:
\begin{align*}
    & M \frac{d\widetilde{\xi}}{dt}(t)+(A+B)\widetilde{\xi}=G(\widetilde{\xi}),\\
    & M\widetilde{\xi}(0)=\widetilde{U}_{0}.
\end{align*}
The matrices $M, A$ are called the mass and stiffness matrices, and they are defined by
    \[\forall 1 \leqslant i,\;j \leqslant {N_{loc}}{N_{el}},\;\;{M_{ij}} = {(\widetilde {{\phi _j}},\;\widetilde {{\phi _i}})_\Omega },\;\;{A_{ij}} = {a_\epsilon }(\widetilde {{\phi _j}},\;\widetilde {{\phi _i}}).\]

The matrix $B$ results from the convective term, and the vector $G(\widetilde{\xi})$ depends on the vector solution
\begin{align*}
    & \forall 1 \leqslant i,\;j \leqslant {N_{loc}}{N_{el}},\;\;\;{(B)_{ij}} = b(c;\;e;\;{\widetilde \phi _j},\;{\widetilde \phi _i}),\\
    & \forall 1 \leqslant i \leqslant {N_{loc}}{N_{el}},\;\;\;{(G)_i} = L(\widetilde \xi ;\;{\widetilde \phi _i}).
\end{align*}
Since the matrix $M$ is invertible and the vector function $G(\widetilde{\xi})$ is Lipschitz with respect to $\widetilde{\xi}$, there exists a unique solution to the variational problem (\ref{variationalforma})-(\ref{variationalformb}).
\section{Error estimates}
In this section, we first present the Gronwall's inequalities \cite{heywood1990finite}, which are important tools for analyzing time-dependent problems.
\begin{lem} [Continuous Gronwall inquality]\label{Gronwall}
Let $f, g, h$ be piecewise continuous nonnegative functions defined on (a, b). Assume that $g$ is nondereasing. Assume that there is a positive constant $C$ independent of $t$ such that
\begin{equation*}
    \forall t\in(a, b), \; f(t)+h(t)\leq g(t)+C\int_{a}^{t}f(s)ds.
\end{equation*}
Then,
\begin{equation*}
    \forall t\in(a, b), \; f(t)+h(t)\leq e^{C(t-a)}g(t).
\end{equation*}
\end{lem}
Now we state a priori error estimates for the semidiscrete scheme \cite{rivi2002non}.
\begin{thm}
Assume that the solution $u$ to problem (\ref{Eulerianmodela})-(\ref{Eulerianmodeld}) belongs to $H^{1}(0, T; H^{2}(\varepsilon_{h}))$ and that $u_{0}$ belongs to $H^{s}(\varepsilon_{h})$ for $s>3/2$. Assume that $\beta_{0}\geq 1$. In the case of SIPG and IIPG, assume that $\sigma_{e}^{0}$ is sufficiently large for all $e$. Then, there is a constant $C$ independent of $h$ such that
\begin{align*}
    &{\left\| {u - {U_h}} \right\|_{{L^\infty }({L^2}(\Omega ))}} + {\left(\int_0^T {\left\| {u(t) - {U_h}(t)} \right\|_\varepsilon ^2dt} \right)^{1/2}}\\
    &\leqslant C{h^{\min (k + 1,\;s) - 1}}({\left\| u \right\|_{{H^1}(0,\;T;\;{H^s}({\varepsilon _h}))}} + {\left\| {{u_0}} \right\|_{{H^s}({\varepsilon _h})}}).
\end{align*}
\end{thm}
\begin{proof}
We omit some details which is similar to the proof of Theorem 3.4.(\cite{Riv08}). We write $u-U_{h}=\rho-\chi$ with $\rho=u-\widetilde{u}$ and $\chi=U_{h}-\widetilde{u}$. The function $\widetilde{u}\in \mathcal{D}_{k}(\varepsilon_{h}))$ is an approximation of $u$ that satisfies good error bounds.
The error equation is satisfied for all $v$ in  $\mathcal{D}_{k}(\varepsilon_{h})$:
\begin{align*}
    (\frac{\partial \chi}{\partial t}, v)_{\Omega}&+a_{\epsilon}(\chi,v)+b(c,e;\chi,v)=(\frac{\partial \rho}{\partial t}, v)_{\Omega}+a_{\epsilon}(\rho,v)\\
                                                  &+b(c,e;\rho,v)+(f(U_{h})-f(u),v)_{\Omega}.
\end{align*}
Now, by choosing $v=\chi$ and using the coercivity of $a_{\epsilon}$, we obtain
\begin{align*}
   &\frac{1}{2}\frac{d}{{dt}}\left\| {\chi} \right\|_{{L^2}(\Omega )}^2 + \kappa \left\| {\chi} \right\|_\varepsilon ^2 +b(c;e;\chi,\chi) \leqslant (\frac{\partial \rho}{\partial t}, \chi)_{\Omega}\\&+a_{\epsilon}(\rho,\chi)+b(c,e;\rho,\chi)+(f(U_{h})-f(u),\chi)_{\Omega}.
\end{align*}
We use Green's formula and the fact that $\nabla  \cdot \left( {\begin{array}{*{20}{c}}
  c \\
  e
\end{array}} \right)=0$:
\begin{align*}
    \sum\limits_{E \in {\varepsilon _h}} {\int_E {\left( {\begin{array}{*{20}{c}}
  c \\
  e
\end{array}} \right)\chi\cdot\nabla\chi}   }&=\frac{1}{2}\sum\limits_{E \in {\varepsilon _h}} {\int_E {\left( {\begin{array}{*{20}{c}}
  c \\
  e
\end{array}} \right)\cdot\nabla\chi^{2}}}\\
& = \frac{1}{2}\sum\limits_{E \in {\varepsilon _h}} \int_{\partial E} {\left( {\begin{array}{*{20}{c}}
  c \\
  e
\end{array}} \right)\cdot {\left( {\begin{array}{*{20}{c}}
  \overrightarrow{n_{1}} \\
    \overrightarrow{n_{2}}
\end{array}} \right)_{E}\chi^{2}}}\\
&=\frac{1}{2}\sum\limits_{e \in {\Gamma _h}} \int_e (c\overrightarrow{n_{1}}+e\overrightarrow{n_{2}})[\chi^{2}].
\end{align*}
Thus we obtain
\begin{align*}
b(c,e;\chi ,\chi)=&- \sum\limits_{E \in {\varepsilon _h}} {\int_E {\left( {\begin{array}{*{20}{c}}
  c \\
  e
\end{array}} \right)\chi\cdot\nabla\chi}   }+\sum\limits_{e \in {\Gamma _h}} \int_e (c\overrightarrow{n_{1}}+e\overrightarrow{n_{2}})\chi^{up}[\chi]\\
& = \sum\limits_{e \in {\Gamma _h}} \int_e (c\overrightarrow{n_{1}}+e\overrightarrow{n_{2}})(\chi^{up}[\chi]-\frac{1}{2}[\chi^{2}])\\
& = \sum\limits_{e \in {\Gamma _h}} \int_e (c\overrightarrow{n_{1}}+e\overrightarrow{n_{2}})(\chi^{up}[\chi]-\{\chi\}[\chi])\\
&=\frac{1}{2}\sum\limits_{e \in {\Gamma _h}} \int_e |c\overrightarrow{n_{1}}+e\overrightarrow{n_{2}}|[\chi]^{2}\geq 0.
\end{align*}
We now bound each term in $b(c,e;\rho,\chi)$. Using Cauchy-Schwarz's and Young's inequalities, we have
\begin{equation*}
 \sum\limits_{E \in {\varepsilon _h}} {\int_E {\left( {\begin{array}{*{20}{c}}
  c \\
  e
\end{array}} \right)\rho\cdot\nabla\chi}   }\leq C \sum\limits_{E \in {\varepsilon _h}}{\left\| \rho \right\|_{L^{2}(E)}}{\left\| {\nabla \chi} \right\|_{L^{2}(E)}}
\leq \frac{\kappa}{8}{\left\| { \chi} \right\|_{\varepsilon}^{2}}+C{\left\| {\rho} \right\|_{L^{2}(\Omega)}^{2}}
\end{equation*}
and
\begin{align*}
 \sum\limits_{e \in {\Gamma _h}} \int_e (c\overrightarrow{n_{1}}+e\overrightarrow{n_{2}})\chi^{up}[\chi]&\leq {\sum\limits_{e\in \Gamma_{h}} {\left\| {{{\left| {c\overrightarrow {{n_1}}  + e\overrightarrow {{n_2}} } \right|}^{\frac{1}{2}}}[\chi]} \right\|} _{0,\;e}}{\left\| {{{\left| {c\overrightarrow {{n_1}}  + e\overrightarrow {{n_2}} } \right|}^{\frac{1}{2}}}\rho_{*}} \right\|} _{0,\;e}\\&\leq \frac{1}{4}{\sum\limits_{e\in \Gamma_{h}} {\left\| {{{\left| {c\overrightarrow {{n_1}}  + e\overrightarrow {{n_2}} } \right|}^{\frac{1}{2}}}[\chi]} \right\|} _{0,\;e}^{2}}+C\sum\limits_{e\in\Gamma_{h}}{\left\| \rho^{up} \right\|} _{L^{2}(e)}^{2}.
\end{align*}
Finally, we bound the nonlinear source term, using the Lipschitz property:
\begin{equation*}
    \int_{\Omega}(f(U_{h})-f(u))\chi\leq C{\left\| ({U_{h}-u}) \right\|}_{L^{2}(\Omega)} {\left\| \chi \right\|}_{L^{2}(\Omega)} \leq C {\left\| \chi \right\|}_{L^{2}(\Omega)}^{2}+C {\left\| \rho \right\|}^{2}_{L^{2}(\Omega)}.
\end{equation*}
The other terms are identical to the ones in the proof of Theorem 2.13 and 3.4 (\cite{Riv08}). Then the main result is obtained by combining all bounds and using Gronwall's inequality of Lemma \ref{Gronwall}.
\end{proof}

We can choose any of the time discretizations such as backward Euler and forward Euler and some that are of high order such as Crank-Nicolson and Runge-Kutta methods. The analysis of the resulting fully discrete schemes can be done in a common way.

\section*{Acknowledgement}
The second author wishes to thank the financial support from the Erasmus Mundus Scholarship of the European Union during his visiting study in Europe.  



\bibliographystyle{model1-num-names}
\bibliography{Xbib}

\begin{thebibliography}{22}
\expandafter\ifx\csname natexlab\endcsname\relax\def\natexlab#1{#1}\fi
\providecommand{\bibinfo}[2]{#2}
\ifx\xfnm\relax \def\xfnm[#1]{\unskip,\space#1}\fi
\bibitem[{Daly and Zannetti(2007)}]{daly2007air}
\bibinfo{author}{A.~Daly}, \bibinfo{author}{P.~Zannetti},
\newblock \bibinfo{title}{Air pollution modeling--an overview}
  (\bibinfo{year}{2007}).
\bibitem[{Blom and Verwer(2000)}]{blomcomparison}
\bibinfo{author}{J.~Blom}, \bibinfo{author}{J.~Verwer},
\newblock \bibinfo{title}{A comparison of integration methods for atmospheric
  transport-chemistry problems},
\newblock \bibinfo{journal}{Journal of computational and Applied Mathematics}
  \bibinfo{volume}{126} (\bibinfo{year}{2000}) \bibinfo{pages}{381--396}.
\bibitem[{Botchev and Verwer(2003)}]{botchev2003new}
\bibinfo{author}{M.~Botchev}, \bibinfo{author}{J.~Verwer},
\newblock \bibinfo{title}{A new approximate matrix factorization for implicit
  time integration in air pollution modeling},
\newblock \bibinfo{journal}{Journal of computational and applied mathematics}
  \bibinfo{volume}{157} (\bibinfo{year}{2003}) \bibinfo{pages}{309--327}.
\bibitem[{Dimov et~al.(2004)Dimov, Georgiev, Ostromsky, and
  Zlatev}]{dimov2004computational}
\bibinfo{author}{I.~Dimov}, \bibinfo{author}{K.~Georgiev},
  \bibinfo{author}{T.~Ostromsky}, \bibinfo{author}{Z.~Zlatev},
\newblock \bibinfo{title}{Computational challenges in the numerical treatment
  of large air pollution models},
\newblock \bibinfo{journal}{Ecological modelling} \bibinfo{volume}{179}
  (\bibinfo{year}{2004}) \bibinfo{pages}{187--203}.
\bibitem[{Zlatev(1995)}]{zlatev1995computer}
\bibinfo{author}{Z.~Zlatev}, \bibinfo{title}{Computer treatment of large air
  pollution models}, \bibinfo{publisher}{Kluwer Academic Publishers},
  \bibinfo{year}{1995}.
\bibitem[{Zlatev and Dimov(2006)}]{zlatev2006computational}
\bibinfo{author}{Z.~Zlatev}, \bibinfo{author}{I.~Dimov},
  \bibinfo{title}{Computational and numerical challenges in environmental
  modelling}, \bibinfo{publisher}{Elsevier Science}, \bibinfo{year}{2006}.
\bibitem[{Zlatev et~al.(1994)Zlatev, Dimov, and Georgiev}]{zlatev1994studying}
\bibinfo{author}{Z.~Zlatev}, \bibinfo{author}{I.~Dimov},
  \bibinfo{author}{K.~Georgiev},
\newblock \bibinfo{title}{Studying long-range transport of air pollutants},
\newblock \bibinfo{journal}{Computational Science and Engineering}
  \bibinfo{volume}{1} (\bibinfo{year}{1994}) \bibinfo{pages}{45--52}.
\bibitem[{Yuan(2000)}]{Yuan2000numerical}
\bibinfo{author}{G.~Yuan},
\newblock \bibinfo{title}{Uniqueness and stability of difference solution with
  nonuniform meshes for nonlinear parabolic systems},
\newblock \bibinfo{journal}{Mathematica numerica sinica} \bibinfo{volume}{22}
  (\bibinfo{year}{2000}) \bibinfo{pages}{139--150}.
\bibitem[{Guo and Zhang(2007)}]{guo2007numerical}
\bibinfo{author}{S.~Guo}, \bibinfo{author}{Z.~Zhang},
\newblock \bibinfo{title}{Numerical methods based on characteristic centered
  finite difference procedure for a class of nonlinear evolution equations},
\newblock \bibinfo{journal}{Chinese Journal of Computational Physics}
  \bibinfo{volume}{24} (\bibinfo{year}{2007}) \bibinfo{pages}{637}.
\bibitem[{Zhou et~al.(1997)Zhou, Shen, and Yuan}]{Zhou1997}
\bibinfo{author}{Y.~Zhou}, \bibinfo{author}{L.~Shen},
  \bibinfo{author}{G.~Yuan},
\newblock \bibinfo{title}{Finite difference method of first boundary problem
  for quasilinear parabolic systems (v)},
\newblock \bibinfo{journal}{Science in China Series A: Mathematics}
  \bibinfo{volume}{40} (\bibinfo{year}{1997}) \bibinfo{pages}{1148--1157}.
\bibitem[{Wang and Zhang(2010)}]{wang10}
\bibinfo{author}{P.~Wang}, \bibinfo{author}{Z.~Zhang},
\newblock \bibinfo{title}{Quadratic finite volume element method for the air
  pollution model},
\newblock \bibinfo{journal}{International Journal of Computer Mathematics}
  \bibinfo{volume}{87} (\bibinfo{year}{2010}) \bibinfo{pages}{2925--2944}.
\bibitem[{Cockburn et~al.(2000)Cockburn, Karniadakis, and Shu}]{Coc00}
\bibinfo{author}{B.~Cockburn}, \bibinfo{author}{G.~E. Karniadakis},
  \bibinfo{author}{C.-W. Shu}, \bibinfo{title}{Discontinuous Galerkin Methods:
  Theory, Compuration and Applications}, \bibinfo{publisher}{Springer},
  \bibinfo{year}{2000}.
\bibitem[{Rivi\`{e}re(2008)}]{Riv08}
\bibinfo{author}{B.~Rivi\`{e}re}, \bibinfo{title}{Discontinuous Galerkin
  Methods For Solving Elliptic and Parabolic Equations: Theory and
  Implementation}, \bibinfo{publisher}{SIAM}, \bibinfo{year}{2008}.
\bibitem[{Arnold(1982)}]{arnold1982interior}
\bibinfo{author}{D.~Arnold},
\newblock \bibinfo{title}{An interior penalty finite element method with
  discontinuous elements},
\newblock \bibinfo{journal}{SIAM Journal on Numerical Analysis}
  (\bibinfo{year}{1982}) \bibinfo{pages}{742--760}.
\bibitem[{Baumann and Oden(1999)}]{baumann1999discontinuous}
\bibinfo{author}{C.~Baumann}, \bibinfo{author}{J.~Oden},
\newblock \bibinfo{title}{A discontinuous hp finite element method for
  convection--diffusion problems},
\newblock \bibinfo{journal}{Computer Methods in Applied Mechanics and
  Engineering} \bibinfo{volume}{175} (\bibinfo{year}{1999})
  \bibinfo{pages}{311--341}.
\bibitem[{Oden et~al.(1998)Oden, Babu\v{s}ka, and
  Baumann}]{Oden:1998:DHF:301787.301788}
\bibinfo{author}{J.~T. Oden}, \bibinfo{author}{I.~Babu\v{s}ka},
  \bibinfo{author}{C.~E. Baumann},
\newblock \bibinfo{title}{A discontinuous hp finite element method for
  diffusion problems},
\newblock \bibinfo{journal}{J. Comput. Phys.} \bibinfo{volume}{146}
  (\bibinfo{year}{1998}) \bibinfo{pages}{491--519}.
\bibitem[{Rivi\`{e}re(2000)}]{riviere2000discontinuous}
\bibinfo{author}{B.~Rivi\`{e}re},
\newblock \bibinfo{title}{Discontinuous galerkin methods for solving the
  miscible displacement problem in porous media},
\newblock \bibinfo{journal}{Ph. D. Thesis, The University of Texas at Austin}
  (\bibinfo{year}{2000}).
\bibitem[{Rivi\`{e}re et~al.(1999)Rivi\`{e}re, Wheeler, and
  Girault}]{riviere1999improved}
\bibinfo{author}{B.~Rivi\`{e}re}, \bibinfo{author}{M.~Wheeler},
  \bibinfo{author}{V.~Girault},
\newblock \bibinfo{title}{Improved energy estimates for interior penalty,
  constrained and discontinuous galerkin methods for elliptic problems. part
  i},
\newblock \bibinfo{journal}{Computational Geosciences} \bibinfo{volume}{3}
  (\bibinfo{year}{1999}) \bibinfo{pages}{337--360}.
\bibitem[{Rivi{\`e}re et~al.(2002)Rivi{\`e}re, Wheeler, and
  Girault}]{rivi¨¨re2002priori}
\bibinfo{author}{B.~Rivi{\`e}re}, \bibinfo{author}{M.~Wheeler},
  \bibinfo{author}{V.~Girault},
\newblock \bibinfo{title}{A priori error estimates for finite element methods
  based on discontinuous approximation spaces for elliptic problems},
\newblock \bibinfo{journal}{SIAM Journal on Numerical Analysis}
  (\bibinfo{year}{2002}) \bibinfo{pages}{902--931}.
\bibitem[{Wheeler and Darlow(1980)}]{wheeler1980interior}
\bibinfo{author}{M.~Wheeler}, \bibinfo{author}{B.~Darlow},
\newblock \bibinfo{title}{Interior penalty galerkin procedures for miscible
  displacement problems in porous media},
\newblock in: \bibinfo{booktitle}{Computational methods in nonlinear mechanics
  (Proc. Second Internat. Conf., Univ. Texas, Austin, Tex., 1979)}, pp.
  \bibinfo{pages}{485--506}.
\bibitem[{Heywood and Rannacher(1990)}]{heywood1990finite}
\bibinfo{author}{J.~Heywood}, \bibinfo{author}{R.~Rannacher},
\newblock \bibinfo{title}{Finite-element approximation of the nonstationary
  navier-stokes problem part iv: Error analysis for second-order time
  discretization},
\newblock \bibinfo{journal}{SIAM Journal on Numerical Analysis}
  (\bibinfo{year}{1990}) \bibinfo{pages}{353--384}.
\bibitem[{Rivi\`{e}re and Wheeler(2002)}]{rivi2002non}
\bibinfo{author}{B.~Rivi\`{e}re}, \bibinfo{author}{M.~Wheeler},
\newblock \bibinfo{title}{Non conforming methods for transport with nonlinear
  reaction},
\newblock in: \bibinfo{booktitle}{Fluid flow and transport in porous media,
  mathematical and numerical treatment: proceedings of an AMS-IMS-SIAM Joint
  Summer Research Conference on Fluid Flow and Transport in Porous Media,
  Mathematical and Numerical Treatment, June 17-21, 2001, Mount Holyoke
  College, South Hadley, Massachusetts}, volume \bibinfo{volume}{295},
  \bibinfo{organization}{Amer Mathematical Society}, p. \bibinfo{pages}{421}.

\end{thebibliography}







\end{document}